\documentclass[a4paper, 11pt]{article}
\usepackage[mathscr]{eucal}
\usepackage{amssymb}
\usepackage{latexsym}
\usepackage{amsthm}
\usepackage{mathrsfs}
\usepackage{amsmath}
\usepackage[dvips]{graphicx}
\usepackage{psfrag}
\usepackage{lscape}

\newcommand{\A}{\boldsymbol{A}}
\newcommand{\E}{\boldsymbol{E}}
\newcommand{\J}{\boldsymbol{J}}
\newcommand{\I}{\boldsymbol{I}}

\newtheorem{theorem}{Theorem}[section]

\newtheorem{lemma}[theorem]{Lemma}
\newtheorem{corollary}[theorem]{Corollary}
\theoremstyle{definition}

\makeatletter
\@addtoreset{equation}{section}

\makeatother

\title{Largest regular multigraphs with three distinct eigenvalues} 

\author{
Hiroshi Nozaki
}
\begin{document}
\maketitle

\renewcommand{\thefootnote}{\fnsymbol{footnote}}
\footnote[0]{
\\ \noindent
{\it Hiroshi Nozaki}: 
	Department of Mathematics Education, 
	Aichi University of Education, 
	1 Hirosawa, Igaya-cho, 
	Kariya, Aichi 448-8542, 
	Japan.
	hnozaki@auecc.aichi-edu.ac.jp.
}

\begin{abstract}
 We deal with connected $k$-regular multigraphs of order $n$ that  has only three distinct eigenvalues. 
In this paper, we study the largest possible number of vertices of such a graph for given $k$. 
For $k=2,3,7$, the Moore graphs are largest.  
For $k\ne 2,3,7,57$, we show an upper bound $n\leq k^2-k+1$, 
with equality if and only if there exists a finite projective plane of order $k-1$ that admits a polarity.  

\end{abstract}

\section{Introduction}
Let $G$ be a connected $k$-regular multigraph $(V,E)$, which may have a loop. 
For $u,v\in V$, let $m(u,v)$ be the number of edges between $u$ and $v$ if $u \ne v$, and the number of loops on $u$ if $u=v$. 
The {\it adjacency matrix} $\A$ of $G$ is defined to be the square matrix
indexed by $V$ whose $(u,v)$ entry is $ 
m(u,v)$ if $\{u,v\} \in E$ and 
$0$ otherwise.  
The eigenvalues of $\A$ are called the eigenvalues of $G$. 
In this paper, we deal with a $k$-regular multigraph $G$ with only 3 distinct eigenvalues. 
Since the degree of the minimal polynomial of $\A$ is $3$, the diameter of $G$ is at most $2$. This implies that the Moore bound $|V| \leq k^2+1$ holds for $k$-regular multigraphs with only 3 distinct eigenvalues. If $G$ attains this bound, $G$ is called a {\it Moore graph}, which is simple. A Moore graph does not exist except  for $(d,k)=(2,2),(2,3),(2,7),(2,57)$ \cite{BI73,D73}. The following Moore graphs uniquely exist: the 5-cycle for $k=2$, the Petersen graph for $k=3$, and the Hoffman--Singleton graph for $k=7$ \cite{HS60}. For $k=57$, the existence of the Moore graph is still open. 
The main problem of this paper is to improve the Moore bound, and to determine the largest $k$-regular multigraph with only 3 distinct eigenvalues for given $k\geq 3$.  
 
A $k$-regular simple graph of order $n$ is called a {\it strongly regular graph} with parameters $(n,k,\lambda, \mu)$ if 
there exist integers $\lambda$ and $\mu$ such that 
  any two adjacent vertices have $\lambda$ common neighbours, and 
 any two non-adjacent vertices have $\mu$ common neighbours. 
If a connected regular simple graph has only 3 distinct eigenvalues, then it is strongly regular. 
If a connected $k$-regular simple graph satisfies that any two adjacent vertices have at least $\lambda$ common neighbours, and 
 any two non-adjacent vertices have at least $\mu$ common neighbours, 
then the order $n$ has the bound  $n\leq k+1+k(k-1-\lambda)/ \mu $ (see \cite{BCNb}). 
Strongly regular graphs are characterized as the graphs that attain this bound. 

The point-line geometry $(\mathcal{P},
\mathcal{L})$ is called a {\it finite projective plane} of order $q$
if $|\mathcal{P}|=|\mathcal{L}|=q^2+q+1$, there exist 
$q+1$ points in each line, and there exist $q+1$ lines 
through each point.
The {\it incidence matrix} of $(\mathcal{P},
\mathcal{L})$ is the matrix indexed by $\mathcal{P}$ and $\mathcal{L}$ whose $(p,l)$ entry is 1 if $p \in l$, and 0 otherwise. 
An isomorphism $\varphi$ from  $(\mathcal{P},\mathcal{L})$ to the dual plane $(\mathcal{L},\mathcal{P})$ is a {\it polarity} if $\varphi$ is an involution. 
We say $(\mathcal{P},\mathcal{L})$ admits polarity if there exists a polarity from  $(\mathcal{P},\mathcal{L})$ to $(\mathcal{L},\mathcal{P})$. The classical finite projective planes admit a polarity. 
A finite projective plane $(\mathcal{P},
\mathcal{L})$ admits a polarity if and only if the incidence matrix of $(\mathcal{P},
\mathcal{L})$ can be symmetric. 
The symmetric incidence matrix of  $(\mathcal{P},\mathcal{L})$ is the adjacency matrix of a $(q-1)$-regular multigraph with only 3 distinct eigenvalues which has loops.  
For $k\ne 2,3,7,57$, we show an upper bound $n \leq k^2-k+1$ for 
$k$-regular  multigraphs of order $n$ with only 3 distinct eigenvalues. The equality holds if and only if 
the adjacency matrix of the graph is the symmetric incidence matrix of a finite projective plane of order $k-1$ that admits a polarity.

The paper is organized as follows. 
In Section~\ref{sec:2}, 
the linear programming bound \cite{N15} is generalized for 
 connected regular multigraphs. We also give a certain improvement of the Moore bound with prescribed distinct eigenvalues. 
In Section~\ref{sec:3}, we prove the upper bound $n \leq k^2-k+1$ for $k\ne 2,3,7,57$. 
In Section~\ref{sec:4}, we show that the existence of a connected $k$-regular multigraph $G$ of order $k^2-k+1$ with only 3 distinct eigenvalues is equivalent to the existence of a finite projective plane $PG(2,k-1)$ that admits a polarity. 

\section{Bounds for regular multigraphs} \label{sec:2}
Let $G$ be a multigraph $(V,E)$.
For $v_j \in V$ and  $e_{j} \in E$, 
a sequence $w_p=(v_0,e_{1},v_1, e_{2}, v_2, \ldots, v_{p-1},e_{p},v_p)$ is a {\it walk} if 
$e_{j}=\{v_{j-1},v_{j}\}$ for each $j \in \{1,\ldots,p\}$. We shortly write a walk 
$w_p=(e_{1},\ldots, e_{p})$.   
 The number $p$ is called the 
{\it length} of a walk.    
A walk $w_p$ is {\it non-backtracking} if 
there does not exist $j \in \{1,\ldots, p-1\}$ such that $e_{j}=e_{j+1}$, or $p=1$.   
A non-backtracking walk $w_p$ is a 
${\it cycle}$ if $v_{0}=v_{p}$ and $v_0,\ldots, v_{p-1}$ are distinct.   
The minimum length of 
cycles in $G$ is called the {\it girth} of $G$. If $G$ has a loop, then the girth of $G$ is 1. 
It is well known that the $(u,v)$-entry of 
$\boldsymbol{A}^i$ is the number of walks of length $i$ from $u$ to $v$. 
A multigraph $G$ is {\it $k$-regular} if 
$\sum_{v \in V} m(u,v)$ is $k$ for each $u \in V$. 

Let $F_i^{(k)}$ denote a polynomial of degree $i$ defined by
\begin{equation*}
F_0^{(k)}(x)=1, \qquad F_1^{(k)}(x)=x, \qquad  F_2^{(k)}(x)=x^2 - k, 
\end{equation*}
and
\begin{equation*}
F_i^{(k)}(x)=x F_{i-1}^{(k)}(x)- (k-1) F_{i-2}^{(k)}(x)
\end{equation*}
for $i\geq 3$. Note that $F_i^{(k)}(k)=k(k-1)^{i-1}$ for $i\geq 1$.

Singleton \cite{S66} proved the following theorem only for $k$-regular simple graphs.
\begin{theorem} \label{thm:red_path}
Let $G$ be a connected $k$-regular multigraph with adjacency matrix $\boldsymbol{A}$.
Then 
the $(u,v)$-entry of $F_i^{(k)}(\boldsymbol{A})$ is the number of non-backtracking walks 
of length $i$ from $u$ to $v$. 
\end{theorem}
\begin{proof}
We use induction on $i$.  Let $b_{uv}^{(i)}$ be the number of non-backtracking walks 
of length $i$ from $u$ to $v$. 
Let $f_{uv}^{(i)}$ be the $(u,v)$-entry of $F_i^{(k)}(\boldsymbol{A})$.  
For $i=1$, the assertion is trivial. 
For $i=2$, the $(u,v)$-entry $a^{(2)}_{uv}$ of $\boldsymbol{A}^2$ is the number of walks of length $2$ from $u$ to $v$. 
A walk that has backtracking must form $(e_i,e_i)$. The assertion follows from $b_{uv}^{(2)}=a^{(2)}_{uv}-k\delta_{uv}$, 
where $\delta$ is the Kronecker delta.   

Suppose $f_{uv}^{(j)}=b_{uv}^{(j)}$ for each $j \in \{1,\ldots,i-1\}$.  
Since $F_i^{(k)}(\boldsymbol{A})=\boldsymbol{A}F_{i-1}^{(k)}(\boldsymbol{A})-(k-1)F_{i-2}^{(k)}(\boldsymbol{A})$,  we have
\begin{align*}
f_{uv}^{(i)}&=\sum_{s \in V} f_{us}^{(1)} f_{sv}^{(i-1)}-(k-1)f_{uv}^{(i-2)}\\ 
&= \sum_{s \in V} b_{us}^{(1)} b_{sv}^{(i-1)}-(k-1)b_{uv}^{(i-2)}.
\end{align*}
The value $\sum_{s \in V} b_{us}^{(1)} b_{sv}^{(i-1)}$ is the number of walks 
$(e_{1},\ldots,e_{p})$ such that $e_{1}=\{u,*\}$, $e_{p}=\{*,v\}$, and 
$(e_{2},\ldots,e_{p})$ is non-backtracking. 
We remove walks that have backtracking, namely the ones satisfying $e_{1}=e_{2}$.
For given non-backtracking walk $(e_{3},\ldots,e_{p})$, the number of 
choices of $e_{1}$ is equal to $k-1$ because $e_{1} \ne e_{3}$. 
Therefore $f_{uv}^{(i)}=b_{uv}^{(i)}$ follows. 
\end{proof}

Let $\I$ denote the identity matrix. 
Let $\J$ denote the matrix whose entries are all $1$. 
In \cite{N15} we proved the following theorem only for $k$-regular simple graphs. 
\begin{theorem} \label{thm:lp}
Let $G$ be a connected $k$-regular multigraph of order $n$ with adjacency matrix $\A$. 
Let $\tau_0,\ldots,\tau_d$ be
the distinct eigenvalues of $\A$, where 
$\tau_0=k$. Let $f(x)$ be the polynomial defined by $f(x)=\sum_{i=0}^s f_i F_i^{(k)}(x)$ with
a positive integer $s$ and real numbers 
$f_0,\ldots,f_s$ such that $f_0>0$, 
$f_i\geq 0$ for each $i \in \{1,\ldots,s\}$.   
If $f(k)>0$ and $f(\tau_j)\leq 0$ for each $j \in \{1,\ldots,d\}$, then
\[
n \leq \frac{f(k)}{f_0}. 
\] 
\end{theorem}
\begin{proof}
Since $\A$ is a real symmetric  matrix, 
we have the spectral decomposition 
$\A=\sum_{i=0}^d \tau_i\E_i$, where $\E_0=(1/n)\J$.     
It follows that 
\begin{equation} \label{eq:pr1}
\sum_{j=0}^d f(\tau_j) \E_j =f(\A)
=\sum_{i=0}^s f_i F_i^{(k)}(\A).
\end{equation} 
Taking the traces in \eqref{eq:pr1}, we have
\begin{multline*}
f(k)={\rm tr}(f(k) \E_0) \geq 
{\rm tr}\left(\sum_{j=0}^d f(\tau_j) \E_j \right)\\
=
{\rm tr}\left(\sum_{i=0}^s f_i F_i^{(k)}(\A) \right) \geq {\rm tr }(f_0 \I)=nf_0, 
\end{multline*}
because $\E_j$ is positive semidefinite, 
and each entry in $F_i^{(k)}(\A)$ is non-negative by Theorem~\ref{thm:red_path}. 
It therefore follows $n\leq f(k)/f_0$. 
\end{proof}

Let $k_i=k(k-1)^{i-1}$ and $k_0=1$.  


\begin{theorem}\label{thm:harm_bound}
Let $G$ be a connected $k$-regular multigraph of order $n$ with adjacency matrix $\A$. 
Let $F(x)$ be the polynomial defined by 
\begin{equation} \label{eq:thm1}
F(x)=\sum_{i=0}^s f_i F_i^{(k)}(x)
\end{equation} 
for 
  some real numbers $f_0,\ldots,f_s$. 
If the entries of $F(\A)$ are all positive, then
\begin{equation} \label{eq:harm_abso}
n \leq \sum_{i\in\{0,\ldots,d\}: f_i>0} k_i. 
\end{equation}
\end{theorem}
\begin{proof}
Since each $(u,v)$-entry of $F(\A)$ is positive, there exists $i \in \{0,\ldots, d\}$ such that 
$f_i>0$ and the $(u,v)$-entry in $F_i^{(k)}(\A)$ is positive. 
For each $u \in V$, the number of non-backtracking walks of length $i$ from $u$ is equal to $k_i$. Thus the number of non-zero entries 
in $F_i^{(k)}(\A)$ is at most $nk_i$. 
Comparing 
the numbers of positive entries in the both sides in \eqref{eq:thm1}, it follows that 
\[
n^2 \leq \sum_{i\in\{0,\ldots,s\}: f_i>0} n k_i.
\]
 This implies the theorem. 
\end{proof}
Let $H_G(x)$ denote the Hoffman polynomial \cite{H63,HM65}
of a regular multigraph $G$, which is the polynomial of least degree satisfying  $H_G(\A)=\J$. 
If the distinct  eigenvalues of $G$ are $\tau_0=k,\tau_1, \ldots,\tau_d$ and the order of $G$ is $n$, then $H_G$ can be expressed by 
\[
H_G(x)=n\prod_{i=1}^d \frac{x-\tau_i}{k-\tau_i}. 
\]
\begin{corollary} \label{coro:2ch}
Let $G$ be a $k$-regular multigraph of order $n$, with only $d+1$ distinct eigenvalues $\tau_0=k,\tau_1,\ldots,\tau_d$.  
Let $F_G(x)$ be the polynomial defined by  $F_G(x)=\prod_{i=1}^d (x-\tau_i)$. Then, 
from the expression $F_G(x)=\sum_{i=0}^d f_i F_i^{(k)}(x)$, it follows that 
$n \leq  \sum_{i\in\{0,\ldots,d\}: f_i>0} k_i$. 

\end{corollary}
\begin{proof}
The polynomial $F_G(x)$ can be expressed by $F_G(x)=(\prod_{i=1}^d (k-\tau_i)/n)H_G(x)$. Therefore, 
 each entry of $F_G(\A)=(\prod_{i=1}^d (k-\tau_i)/n)\J$ is positive. 
Applying Theorem~\ref{thm:harm_bound} to $F_G(x)$, we obtain the bound  $n \leq  \sum_{i\in\{0,\ldots,d\}: f_i>0} k_i$.
\end{proof}
If each $f_i$ is positive in Corollary~\ref{coro:2ch}, then the bound \eqref{eq:harm_abso} coincides with the Moore bound. 

\section{Upper bound for regular multigraphs with three eigenvalues} \label{sec:3}
In this section, we prove an upper bound for $k$-regular multigraphs with only $3$ distinct eigenvalues, which means  
Theorem~\ref{thm:bound}. First we prove several lemmas to prove Theorem~\ref{thm:bound}. 

\begin{lemma} \label{lem:harm}
Let $G$ be a connected $k$-regular multigraph of order $n$ with only 3 distinct eigenvalues $k$, $\tau_1$, $\tau_2$. If $\tau_1+\tau_2 \geq 0$, then 
$n \leq k^2-k+1$. 
\end{lemma}
\begin{proof}
The polynomial $F_G(x)=(x-\tau_1)(x-\tau_2)$ can be expressed by 
\[
F_G(x)=F_2^{(k)}(x)-(\tau_1+\tau_2) F_1^{(k)}(x) + (k+\tau_1\tau_2)F_0^{(k)}(x). 
\]
By $\tau_1+\tau_2\geq 0$ and Corollary~\ref{coro:2ch},  we have 
$
n \leq k_0+k_2=k^2-k+1  
$. 
\end{proof}

\begin{lemma} \label{lem:no_multi}
In a multigraph of maximum degree at most $k$, if a vertex $u$ is incident
with a multiedge then there are at most $k^2-k$ vertices within distance two of $u$.  
\end{lemma}
\begin{proof} 
Let $v$ be a vertex adjacent to $u$ with a multiedge. 
Then, it follows that 
\begin{align*}
|\{w \in V \colon\,  \partial(u,w) \leq 2\}|&=1+|\{w \in V \colon\, \partial(u,w)=1 \}|+|\{w \in V \colon\, \partial(u,w)=2 \}|\\
&\leq 1+(k-1)+|\{w \in V \colon\, \partial(u,w)=2, (v,w)\in E \}|\\
&\qquad \qquad +|\{w \in V \colon\, \partial(u,w)=2, (v,w)\notin E \}|\\
&\leq 1+(k-1)+(k-2)+(k-1)(k-2)=
k^2-k,
\end{align*}
where  $\partial(u,w)$ is the distance between $u$ and $w$. 
\end{proof}
Let $l_v$ denote the number of loops of $v \in V$. 
\begin{lemma}\label{lem:larger}
Let $G$ be a connected $k$-regular multigraph of order $n$ with only 3 distinct eigenvalues. 
If $n>k^2-k+1$, then $G$ is simple and strongly regular. 
\end{lemma}
\begin{proof}
It suffices to show that $G$ is simple. 
Let $\tau_1$, $\tau_2$ be the distinct eigenvalues of $G$ with $\tau_1,\tau_2 \ne k$.  
By Lemma~\ref{lem:harm}, 
we have $\tau_1+\tau_2<0$. 
By Lemma~\ref{lem:no_multi}, $G$ has no multiedge.  
The Hoffman polynomial of $G$ can 
be expressed by 
\[
H_G(x)=n \frac{(x-\tau_1)(x-\tau_2)}{(k-\tau_1)(k-\tau_2)}. 
\]
It therefore follows that 
\begin{equation}\label{eq:hoff}
n(\A^2-(\tau_1+\tau_2)\A+\tau_1\tau_2 \I)=(k-\tau_1)(k-\tau_2) \J, 
\end{equation}
where $\A$ is the adjacency matrix of $G$. 
Comparing the $(v,v)$-entry of the both sides in \eqref{eq:hoff}, 
we obtain 
\begin{equation*} 
l_v^2-(\tau_1+\tau_2+1)l_v=\frac{1}{n}(k-\tau_1)(k-\tau_2)-k-\tau_1\tau_2.
\end{equation*}
The value $l_v^2-(\tau_1+\tau_2+1)l_v$ is constant for each $v \in V$. 
If $l_v>0$ for each $v \in V$, then 
\[
n \leq 1+(k-2)+(k-2)(k-2)=k^2-3k+3<k^2-k+1, 
\]
which contradicts our assumption. 
We may suppose some $v\in V$ satisfies $l_v=0$. This implies that $l_u^2-(\tau_1+\tau_2+1)l_u=0$, namely $l_u=0$ or $l_u=\tau_1+\tau_2+1$ for each $u \in V$.
Since $\tau_1+\tau_2<0$ holds, it follows that $l_u=\tau_1+\tau_2+1<1$ and $l_u=0$ for each $u \in V$. 
\end{proof}

\begin{lemma}\label{lem:larger_srg}
Let $G$ be a connected $k$-regular multigraph of order $n$ with only 3 distinct eigenvalues. 
If $n> k^2-k+1$ and $k\geq 3$, then there does not exist $G$ except for Moore graphs.  If $n>k^2-k+1$ and $k=2$, then $G$ is the cycle graph of order 4 or 5.  
\end{lemma}
\begin{proof}
By Lemma~\ref{lem:larger}, 
$G$ is strongly regular, and let $(n,k,\lambda,\mu)$ be the parameters of $G$.   
The assertion clearly holds for $k=2$. Suppose $k\geq 3$. 
Let $\tau_1$, $\tau_2$ be the distinct eigenvalues of $G$ with $\tau_1,\tau_2 \ne k$. 
 For connected strongly regular graphs,  it follows that $\mu \ne 0$. 
If $\mu \geq 2$, then
\begin{equation} \label{eq:srg}
n
=k+1+\frac{k^2-\lambda k-k}{\mu}
\leq  \frac{k^2}{2}+\frac{k}{2}+1 \leq k^2-k+1
\end{equation} 
from $k\geq 3$. 
Thus $\mu=1$. If $\lambda=0$, then 
$G$ is a Moore graph.
If $\lambda=1$, then $G$ gives rise to a projective plane with a polarity containing no absolute points,
which is not possible \cite{DF01}. 
If $\lambda >1$, then there exists an integer $s$ such that $k=s(\lambda+1)$ and $n=1+s(\lambda+1)+s(s-1)(\lambda+1)^2$~\cite{DF01}, which gives
\[
n
= 1+k+k^2- s(\lambda+1)^2
\leq 1+k+k^2-3k<k^2-k+1. \qedhere
\]
\end{proof}

\begin{theorem} \label{thm:bound}
Let $G$ be a connected $k$-regular multigraph of order $n$ with  only 3 distinct eigenvalues. 
Then, one has 
$
n \leq k^2-k+1
$ 
for $k\ne 2,3,7,57$. 
\end{theorem}
\begin{proof}
By Lemma~\ref{lem:larger_srg}, if $n > k^2-k+1$, then $G$ is a Moore graph. 
There does not exist a Moore graph except for $k\in \{2,3,7,57\}$ \cite{BI73,D73}. This implies the theorem. 
\end{proof}

\section{Largest regular multigraphs with three eigenvalues} \label{sec:4}

For $k\ne 2,3,7,57$, we have $n\leq k^2-k+1$ by Theorem~\ref{thm:bound}. 
The largest multigraphs are constructed from finite 
projective planes. 
Refer to \cite{Sb} for projective planes. 
Suppose $q=k-1$ is a prime power. Let $\mathbb{F}_q$
 be the finite field of order $q$. 
Let $V_q$ be a 3-dimensional vector space over $\mathbb{F}_q$. Let $\mathcal{P}_q$ ({\it resp.}\ $\mathcal{L}_q$) be the set of 
all 1-dimensional ({\it resp.}\ 2-dimensional) subspaces of $V_q$. 
Note that $|\mathcal{P}_q|=|\mathcal{L}_q|=q^2+q+1=k^2-k+1$.  
A point $p \in \mathcal{P}_q$ is  incident with a line  
$l \in \mathcal{L}_q$ if $p \subset l$. 
The point-line geometry $(\mathcal{P}_q,
\mathcal{L}_q)$ is called a {\it classical} finite projective plane. 
Let $\Gamma_q$ denote the incidence graph  of $(\mathcal{P}_q,\mathcal{L}_q)$. The graph $\Gamma_q$ is bipartite and its adjacency matrix 
can be expressed by 
\[
\begin{pmatrix}
O& \A\\
\A^{\top} &O
\end{pmatrix},
\]
where $\A$ is the incidence matrix of  $(\mathcal{P}_q,\mathcal{L}_q)$. 
The set of eigenvalues of $\Gamma_q$ is $\{\pm (q+1),  \pm \sqrt{q}\}$. 
We may suppose $\A$ is symmetric by the correspondence  $\{(p,l) \in \mathcal{P}_q \times  \mathcal{L}_q \colon\, p \perp l \}$, where we use the usual inner product of $V_q$. This implies that $\A$ is  
the adjacency matrix of a $(q+1)$-regular graph $G_q$ and has only 3 distinct eigenvalues $\{q+1, \pm \sqrt{q}\}$. Note that $G_q$ has loops.  For any prime power $q$, the graph $G_q$ is a largest $k$-regular multigraph attaining the bound from Theorem~\ref{thm:bound}.

The following is a necessary condition for a graph to attain the bound from Theorem~\ref{thm:bound}. 

\begin{lemma} \label{prop:1}
Let $G$ be a connected $k$-regular multigraph of order $n$ with only 3 distinct eigenvalues $k$, $\tau_1$, $\tau_2$. 
If $n=k^2-k+1$, then $G$ has a loop and no multiedge, $l_v \in \{0,1\}$ for each $v \in V$, and $\tau_1+\tau_2=0$. 
\end{lemma}
\begin{proof}
By $n=k^2-k+1$ and Lemma~\ref{lem:no_multi},  there does not exist a multiedge in $G$. 
If there exists $v \in V$ such that 
$l_v>1$, then 
\[
n \leq 1+(k-2)+(k-2)(k-1)=k^2-2k+1<k^2-k+1. 
\]
Thus $l_v\leq 1$ for each $v\in V$. 
As we see in the proof of Lemma~\ref{lem:larger}, 
there exists $v\in V$ such that $l_v=0$. 
Moreover $l_u^2-(\tau_1+\tau_2+1)l_u=0$, namely $l_u=0$ or $l_u=\tau_1+\tau_2+1$ for each $u \in V$. 
If there exists $u\in V$ such that $l_u=\tau_1+\tau_2+1=1$, then $\tau_1+\tau_2=0$. 
Assume $l_u=0$ for each $u \in V$. Now $G$ is a strongly regular graph with parameters $(v,k,\lambda,\mu)$. 
If $\mu \geq 2$, then \eqref{eq:srg} holds. 
The last equality in \eqref{eq:srg} is attained only for  $(n,k)=(7,3)$, which is impossible.  
Thus $\mu=1$. 
By the same argument as the last part in the proof of 
Lemma~\ref{lem:larger_srg}, for any $\lambda$ there does not exist $G$ of order $k^2-k+1$.   
\end{proof}

The following is the main theorem in this section. 
\begin{theorem}\label{thm:4.2}
The existence of a connected $k$-regular multigraph $G$ of order $k^2-k+1$ with only 3 distinct eigenvalues is equivalent to the existence of a finite projective plane $PG(2,k-1)$ that admits a polarity. 
\end{theorem}
\begin{proof}
If a finite projective plane $PG(2,k-1)$ that admits a polarity exists, then the incidence matrix can be symmetric, and it is the adjacency matrix of a $k$-regular multigraph of order $k^2-k+1$ with only 3 distinct eigenvalues. 

Let $G$ be  a connected $k$-regular multigraph of order $k^2-k+1$ with only 3 distinct eigenvalues. 
By Lemma~\ref{prop:1}, the eigenvalues are $k,\pm \tau$, and the bipartite double graph $G'$ of $G$ is 
 simple. 
Since the eigenvalues of $G'$ are $\pm k, \pm \tau $,  the diameter of $G'$ is at most $3$. Thus the graph $G'$ attains the bipartite Moore bound 
$
n \leq 2(1+(k-1)+(k-1)^2)=2(k^2-k+1),
$
and the girth of $G'$ is $6$. The graph $G'$ is the cage $v(k,6)$, and $G'$ must be the incidence graph of a finite projective plane $PG(2,k-1)$ (see \cite[Section 6.9]{BCNb}). Now the incidence matrix of the projective plane $PG(2,k-1)$ is symmetric, and hence there exists a polarity on it. 
\end{proof}
By Theorem~\ref{thm:4.2}, 
largest $k$-regular multigraphs with only 3 distinct eigenvalues are obtained for a prime power $q=k-1$. 
Open cases of small degrees are $k=11, 13, 15, 16, 19, 21, 22, 23, \ldots $. 
For $q\equiv 1,2 \pmod 4$, if a
projective plane of order $q$ exists, then $q$ is the sum of
two integral squares \cite{BR49}. Therefore for $k=13$ a projective plane of order $14$ does not exist. 
For $k=11$, there does not exist a finite projective plane of order $10$ by a computer search \cite{L91}.  If $\A$ is the adjacency matrix of some $k$-regular multigraph, 
then $\A+t \I$  is that of a $(k+t)$-regular multigraph, and has the same number of distinct eigenvalues as $\A$. 
This construction gives a candidate of the largest graph when 
a projective plane does not exist.

\bigskip

\noindent
\textbf{Acknowledgments.} 
The author is supported by JSPS KAKENHI Grant Numbers 16K17569, 26400003, 17K05155,  18K03396, and 19K03445. 
The author would like to thank the four anonymous referees for their valuable suggestions which helped to improve the earlier version of this paper.  
	
\end{document}